\documentclass[
  a4paper, 
  reqno, 
  oneside, 
  11pt
]{amsart}

\usepackage[dvipsnames]{xcolor} 
\colorlet{cite}{LimeGreen!10!Green}
\usepackage{tikz}  
\usetikzlibrary{arrows,positioning,patterns}
\tikzset{ 
  baseline=-2.3pt,
  text height=1.5ex, text depth=0.25ex,
  >=stealth,
  node distance=2cm,
  mid/.style={fill=white,inner sep=2.5pt},
}
\usepackage[none]{hyphenat}
\usepackage{lmodern}
\usepackage{tikz-cd}
\usepackage{amsthm, amssymb, amsfonts}

\usepackage{graphicx,caption,subcaption}
\usepackage{braket}  
\usepackage{microtype}

\usepackage[%
  bookmarks=true,			
  unicode=true,			
  pdftitle={Deformations of noncompact manifolds},		%
  pdfauthor={Elizabeth Gasparim}{Francisco Rubilar},	%
  pdfkeywords={Hodge structures}{adjoint orbit}{compactification}{Lefschetz fibration},	
  colorlinks=true,		
  linkcolor=Blue,			
  citecolor=cite,		
  filecolor=magenta,		
  urlcolor=black		
]{hyperref}				
\usepackage{cleveref}

\newtheoremstyle{mydef}
  {}		
  {}		
  {}		
  {}		
  {\scshape}	
  {. }		
  { }		
  {\thmname{#1}\thmnumber{ #2}\thmnote{ #3}}	

\theoremstyle{plain}	
\newtheorem{theorem}{Theorem} 
\newtheorem{lemma}[theorem]{Lemma} 
\newtheorem{proposition}[theorem]{Proposition}
\newtheorem*{theorem*}{Theorem}

\theoremstyle{mydef} 
\newtheorem{definition}[theorem]{Definition}
\newtheorem*{conjecture*}{Conjecture}
\theoremstyle{remark}
\newtheorem{remark}[theorem]{Remark}

\newtheorem{example}[theorem]{Example}

\DeclareMathOperator{\Proj}{Proj}

\newtheorem*{proposition*}{Proposition}
\newtheorem*{lemma*}{Lemma}
\newtheorem*{corollary*}{Corollary}
\theoremstyle{definition}
\theoremstyle{remark}

\newtheorem{question}[equation]{\bf Question}

\DeclareMathOperator{\Tot}{Tot}

\newcommand{\ce}{\mathrel{\mathop:}=}

\newcommand{\enm}[1]{\ensuremath{#1}}          %

\newcommand{\cal}[1]{\mathcal{#1}}

\renewcommand{\bar}[1]{\overline{#1}}

\newcommand{\CC}{\enm{\mathbb{C}}}

\newcommand{\NN}{\enm{\mathbb{N}}}
\newcommand{\RR}{\enm{\mathbb{R}}}

\newcommand{\PP}{\enm{\mathbb{P}}}

\newcommand{\Mm}{\enm{\cal{M}}}
\newcommand{\Nn}{\enm{\cal{N}}}
\newcommand{\Oo}{\enm{\cal{O}}}

\newcommand{\Rr}{\enm{\cal{R}}}

\newcommand{\Ww}{\enm{\cal{W}}}
\newcommand{\Xx}{\enm{\cal{X}}}

\renewcommand{\phi}{\varphi}
\renewcommand{\theta}{\vartheta}
\renewcommand{\epsilon}{\varepsilon}

\renewcommand{\to}[1][]{\xrightarrow{\ #1\ }}

\newcommand{\old}[1]{}

\title[20 open questions about deformations]{20 open questions about deformations of compactifiable manifolds}
\author{Edoardo Ballico,  Elizabeth Gasparim, {\tiny and}  Francisco Rubilar}
\address{Ballico - Dept. Mathematics, University of Trento, I-38050 Povo, Italy.\newline
Gasparim, Rubilar -  Depto. Matem\'aticas, Universidad Cat\'olica del Norte, Antofagasta, Chile. 
\newline
 ballico@science.unitn.it, etgasparim@gmail.com, rubilar\_n17@hotmail.com}

\begin{document}

\begin{abstract}
Deformation theory of  complex manifolds is a classical subject with recent new advances in the noncompact case using both algebraic and analytic methods.

In this note we recall some concepts of the existing theory and introduce new notions of deformations for manifolds with boundary, for compactifiable manifolds, and for $q$-concave spaces. We highlight some of the possible applications and give a list of open questions which we intend as a guide for further research in this rich and beautiful subject.
\end{abstract}
\maketitle
\tableofcontents

\section{Motivation}

There exists a beautiful and successful theory of deformations of complex structures for 
compact manifolds, as described in \cite{kod} (see also \cite{man,hart2}).
Here we will discuss some aspects of deformation theory of noncompact manifolds.
A  short recent survey paper in this area appears in \cite{GR}.
Many applications to mathematical physics require the use of noncompact manifolds, 
and some of these were our initial motivation to consider noncompact spaces. 
 For example, in 4 dimensions, the Nekrasov instanton partition function is defined over
noncompact spaces and uses equivariant integration, which would evaluate trivially to zero  
if considered over compact spaces. Similarly, in 6 dimensions the theory of BPS states 
considers integration over  moduli of bundles on noncompact threefolds. 
There are also various purely  mathematical 
reasons why one might need deformations of complex structures of noncompact 
manifolds. We are especially interested in toric Calabi--Yau threefolds, 
and these must necessarily be noncompact. Furthermore, motivated by questions coming from 
mirror symmetry, we also wish to consider total spaces of cotangent bundles and 
some of their deformations such as  noncompact
semisimple adjoint orbits. These are just a few among  many reasons 
to develop deformation theory on the noncompact case. 

Nevertheless, compact complex manifolds (or compact complex spaces) have several good properties, 
which noncompact ones lack. For the deformation theory of compact complex manifolds as 
in the papers by Kodaira and Spencer more than $65$ years ago 
compactness was essential 
to get global solutions through the use of   elliptic  PDE's. In all other approaches to deformation theory, even 
for singular complex spaces, compactness is crucial for deciding  existence questions.
Thus, we may expect to find additional difficulties when dealing with the noncompact case. 

There are many possible choices of how to go about generalizing deformation theory.
Here we propose some new concepts, giving examples. Our main goal 
 here is to motivate other researchers to think about some of  the beautiful problems on 
 deformation theory, and to such end we will state several open questions
 which  we hope shall serve as motivation.

We will consider  deformations of manifolds with boundary, compactifiable manifolds and $q$-concave manifolds. For each of these classes of spaces, we introduce notions of deformations that are interesting to study, and state open questions about them. We also include a section where we relate deformations with cobordisms and we propose a potential new approach to study both themes in a common setting.

\section{Submersions}

The classical Ehresmann fibration theorem says that if $f\colon M \rightarrow N$ is a proper submersion between smooth 
manifolds (possibly with boundary), then it is a smooth fiber bundle; hence smoothly locally trivial.  
The holomorphic analogue is totally false and 
deformation theory might be described as the study of this failure.

Let us first note that if a holomorphic map is a smooth submersion then it is also a holomorphic submersion. We now 
formulate  this precisely. 
Consider a holomorphic map  $f\colon  X\to Y$  between complex manifolds. Fix $x\in X$ and set $y\ce f(x)$. Let $n$ (resp. $m$) be
the dimension of $X$ at $x$ (resp. of $Y$ at $y$). 
Let $T_{x}X$ (resp. $T_{y}Y$) denote the tangent space at $x$ (resp. $y$)
of the complex manifold $X$ (resp. $Y$). 
Then $T_{x}X$ (resp. $T_{y}Y$) is a complex vector space of dimension $n$ (resp. $m$).
Let
$\Theta _{x}X$ denote the real tangent space of
$X$ at
$x$ (resp.
$Y$ at
$y$) seen as a $2n$-dimensional (resp. $2m$-dimensional) differentiable manifold. We have $\dim _{\RR} \Theta _{x}X =2n$,
$\dim _{\RR} \Theta _{y}Y =2m$, $$\Theta _{x}X\otimes _{\RR}\CC \cong T_{x}X\oplus \overline{T_{x}X} \quad \text{and}
\quad \Theta _{y}Y\otimes _{\RR}\CC \cong T_{y}Y\oplus \overline{T_{y}Y},$$
 where $\overline{V}$ denote the complex conjugate
of the complex vector space $V$. 

\begin{lemma} If $f\colon X\rightarrow Y$ is a submersion at $x$ of differentiable manifolds, i.e. if the
real differential $\theta _x\colon  \Theta _{x}X\to \Theta _{y}Y$ is surjective, then $f$ is a submersion at $x$ as
complex manifolds, 
i.e. the $\CC$-linear map $df_x\colon T_{x}X \to T_{y}Y$ is surjective. 
\end{lemma}

\begin{proof}
The surjectivity of  the real differential $\theta _x: \Theta _{x}X\to \Theta _{y}Y$ 
implies the surjectivity of the $\RR$-linear map 
$\theta _{x\otimes \CC} \colon \Theta _{x}X\otimes _{\RR}\CC \to \Theta
_{y}Y\otimes _{\RR}\CC$, i.e. of the map $df_x\oplus \overline{df_x}: T_{x}X\oplus \overline{T_{x}X}
\to T_{y}Y\oplus \overline{T_{y}Y}$. To prove our claim it is sufficient to prove that the real rank of $df_x\oplus
\overline{df_x}$ is twice the complex rank of $df_x$. 
Fixing a basis, we can see the linear map $df_x\oplus \overline{df_x}$ as a diagonal matrix. Denote by $J$ 
the complex Jacobian of $f$, since we are assuming $f$ holomorphic, denoting the real Jacobian of $f$ by $J_{\mathbb{R}}$ we have
$$J_{\mathbb{R}}=\begin{bmatrix}
J&0\\
* &\overline{J}
\end{bmatrix}$$
Since $\theta _x$ is surjective, the rank of $J_{\mathbb{R}}$ is $2n$. So we have 
$$2n=\mathrm{rk}(J)+\mathrm{rk}(\overline{J})=2\mathrm{rk}(J),$$
hence real rank of $df_x\oplus
\overline{df_x}$ is twice the complex rank of $df_x$. Therefore, $J$ has rank $n=\mathrm{dim}T_{y}Y$ 
which is maximal, thus we conclude that $f$ is a submersion as complex manifolds.
\end{proof}

\begin{example} Here is a classic example of a holomorphic submersion 
that is not holomorphically locally trivial. Let $S= \mathbb C \setminus \{0,1\}$, and 
 consider the family $X$ of elliptic curves in $\mathbb P^2$ parametrized by $\lambda \in S$ given by 
 $$y^2z= x(x-z)(x-\lambda z).$$ Then the corresponding proper submersion 
$$ \begin{array}{rccl} \pi \colon &X &\rightarrow & S\\
                             & (\lambda, [x:y:z]) &  \mapsto & \lambda
\end{array}$$
is certainly not locally trivial because its fibers are not isomorphic. 
\end{example}


\section{Deformations of manifolds with boundary}

Fix $r\in (\NN\setminus \{0\})\cup \{+\infty\}$. Let $M$ and $N$ be Hausdorff $C^r$-manifolds with a countable basis for their
topology. Let $C^r(M,N)$ denote the set of all $C^r$-maps $f\colon  M\to N$. 
There are two useful topologies (the weak and the
strong topology) on the set $C^r(M,N)$ \cite[Ch. 2, \S\thinspace1]{hir}, which however are the same when $M$ is compact \cite[last 4
lines of page 35]{hir}. Since in this section we consider manifolds with boundary, 
we only need the case in which $M$ is compact, thus we  describe the weak topology (also called
the compact-open topology for $f$ and its derivatives). Let $(\phi,U)$ and $(\psi ,V)$ be charts on $M$ and $N$ respectively.
Let $K\subset U$ be a compact set such that $f(K)\subset V$. Fix a positive real number $\epsilon$ and a non-negative integer
$k$. Let $\vert \ \ \vert$ denote the euclidean norm on $\psi(V)$.
Let $\Nn (f; (U,\phi),(V,\psi),\epsilon,k)$ denote the set of all $g\in C^r(M,N)$ such that 
$$\vert D^i(\psi{f}\phi^{-1})(x) -(\psi{g}\phi^{-1})(x)\vert < \epsilon$$for all $x\in \phi(K)$ and all partial derivatives $D^i$ of order
$i\le k$. 

The {\it weak topology} on $C^r(M,N)$ is the topology for which the collection of 
all sets $\Nn (f; (U,\phi),(V,\psi),\epsilon,k)$ is
prescribed to be a subbasis, i.e. an open set for this topology is an arbitrary union of a finite intersection of sets $\Nn
(f; (U,\phi),(V,\psi),\epsilon,k)$. In the remaining of this section we use the weak topology on $C^r(M,N)$.

There is also the notion 
of $C^r$ manifold with a boundary $(M,\partial M)$ and a topology for the family of all $C^r$-maps $C^r((M,\partial M);(N,\partial N))$
between two manifolds with boundaries \cite[Ch. 2,\S\thinspace3]{hir}. We use these notions to define a $C^r$-locally trivial family of complex manifolds with boundary.
However, as in \cite[pp.~30-31]{hir} we may  also consider  maps from 
 a manifold $(M,\partial M)$ with $\partial M\ne \emptyset$ to a manifold $Y$ with $\partial Y =\emptyset$.

\begin{definition}\label{def-edo}
Let $X$, $Y$ be $C^\infty$ manifolds and $f\colon
 X\to Y$ a $C^\infty$-map. We say that $f$ is a {\it deformation}
  of connected $n$-dimensional complex manifolds if:
  \begin{itemize}
  \item  $X$ and $Y$ are connected, and $\dim X=\dim Y+n$ with $n>0$ 
  \item $f$ is a submersion with connected fibres  $\{M_t\}_{t\in Y}$
  \item  there exists  a countable open covering $\{U_i\}_{i\ge 1}$ of $X$ such that
  for each integer $i\ge 1$ there exists
a collection of $n$ $C^\infty$ functions $$z_i^1(q),\dots ,z_i^n(q)\colon U_i\to \CC$$ such that for each $t\in Y$, the assignment
$$\{q\mapsto (z_i^1(q),\dots ,z_i^n(q))_{\mid U_i\cap f^{-1}(t)}\}$$ provides a complex structure on $M_t$. 
\end{itemize}
\end{definition}

\begin{remark}
Contrary to \cite[p.\thinspace 184]{kod} we do not assume that $f$ is proper, i.e. we do not assume that each $M_t$ is compact.
\end{remark}

\begin{definition}
We say that $f$ is \emph{$C^\infty$ locally trivial on the base} if for each $t\in Y$ there is an open neighborhood $U$ of $t\in Y$
and a diffeomorphism $v\colon 
 f^{-1}(U) \to U\times M_t$ such that $\pi _1\circ v = f_{|f^{-1}(U)}$, where $\pi _1\colon U\times M_t\to U$ is the projection onto the first factor.
\end{definition}

Next we recall the definition of deformation that appears in \cite{GR}.

\begin{definition}{\cite[Def.\thinspace9]{GR}}
	\label{deformation}
	A {\it deformation family} of a complex manifold $X$ is a holomorphic surjective submersion $\tilde{X}\stackrel{\pi}{\rightarrow}D$, where $D$ is a complex disc centred at $0$ (possibly a vector space, possibly infinite dimensional), satisfying:
	\begin{itemize}
		\item $\pi^{-1}(0)=X$,
		\item $\tilde{X}$ is locally trivial  in the $C^{\infty}$ category.
	\end{itemize}
	The fibres $X_t=\pi^{-1}(t)$ are called \textit{deformations} of $X$.
\end{definition}

In further   generality, allowing the parameter space to be a variety or a scheme, and  requiring that the bundle be locally trivial,
we obtain the concept of a {\it family} of complex structures.

\begin{remark}
Note that in Def.\thinspace{\ref{def-edo}} we are not assuming that $f$ is locally trivial in the $C^{\infty}$-category. Def.\thinspace{\ref{def-edo}} is not the same that \cite[Def.\thinspace9]{GR}. Indeed, \cite[Ex.\thinspace7]{GR} 
shows that these two definitions are not equivalent in full generality, but we propose the following question.
\end{remark}

\begin{question}
Under what conditions we can say that the two notions of deformations mentioned above are equivalent, or one implies the other? 
\end{question}


\subsection{Deformations and the $h$-cobordism theorem}\label{hcob}

The fundamental  idea behind cobordism theory is that 2 manifolds are cobordant if their disjoint union 
forms the boundary of some manifold. 
Since a manifold $M$ and any of its deformations $M'$ occur in a family, by considering the family 
restricted to  a path on the base connecting the fibers $M$ to $M'$,  
we see  that these are  cobordant.

In this section, we suggest a relation between  deformations of complex manifolds with boundary 
and  the very well known $h$-cobordism theorem. This is just one of many possible connections  
among the concepts of deformations and cobordisms. 
 We quote from the introduction of Stong \cite{St}: {\it As in any subject, the primary problem is 
 classification of the objects within isomorphism and determination of effective and computable invariants
 to distinguish the isomorphism classes}. In our setup, the notion of cobordant manifolds can be 
 regarded as such a computable invariant.
 
  It is worth mentioning that in  various setups (but not all)
  deformation equivalent implies diffeomorphic, therefore homeomorphic. So, one might wish to use the notions of 
  diffeomorphic or homeomorphic manifolds instead of cobordant. However, it is well known that deciding if 2 
  smooth manifolds are 
  homeomorphic is not computable in the following precise sense: It has been known for a long time 
 that for any finitely presented group $S$, it is possible to 
  construct a 4 dimensional real manifold $ M(S)$  with fundamental group $S$ in such a way that 
  $M(S)$ and $M(R)$ will be homeomorphic if and only if $S$ and $T$ are isomorphic groups. However, 
  the word problem of deciding when 2 finitely presented groups are isomorphic is unsolvable   \cite{Ma}.
Cobordism, being a milder invariant, is more computable.

 We recall some basic ideas and fundamental results about cobordisms and refer the readers to the 
 beautiful book by Stong \cite{St} and references therein for the general theory.

\begin{definition}
An $(n+1)$-dimensional 
{\it cobordism} is a quintuple $ (W;M,N,i,j)$ consisting of an $(n+1)$-dimensional compact differentiable manifold with boundary $W$; closed $n$-manifolds $M, N$ and embeddings $i\colon M\hookrightarrow \partial W$ and  $j\colon N\hookrightarrow \partial W$ with disjoint images such that
$$ \partial W=i(M)\sqcup j(N).$$
\end{definition}

There are many types of cobordisms such as: oriented, unoriented, framed, real, complex, relative, unitary,  special unitary, symplectic, Lagrangian,
exact, $h$-cobordisms,
s-cobordisms,
and so on. 
Here, the choice of to focus on $h$-cobordisms simply reflects the preference of the second author when typing this paper, but certainly many 
other types are of interest. Without making any attempt to review the vast literature on the subject, we  
just point out  one example of application of the concept complex cobordisms, which
 follows from 
the more general result for weakly complex manifolds \cite[Thm.1'']{Mi1} proved by Milnor and Novikov:

\begin{proposition}
If a complex manifold is the boundary of another complex manifold, then all of its Chern numbers are zero. 
\end{proposition}

We now discuss $h$-cobordisms. For a classical  introduction to this subject, we suggest \cite{Mi1, Mi2}.
A brief overview can be found in the $h$-cobordism wikipedia page \url{https://en.wikipedia.org/wiki/H-cobordism}.

\begin{definition}
	An $(n+1)$-dimensional cobordism $ (W;M,N,i,j)$ is an {\it $h$-cobordism}
	 if the inclusions $i\colon M\hookrightarrow \partial W$ and  $j\colon N\hookrightarrow \partial W$ are homotopy equivalences.
\end{definition}

We recall some fundamental and well known results in low dimensions concerning  $h$-cobordisms and then,
 passing to dimension at least 5, we  
state the $h$-cobordism theorem,  proved  in 1962 by S. Smale \cite{S2} expanding on  ideas of his proof (just one year earlier) 
of the Poincar\'e conjecture in dimension greater than or equal to $5$ \cite{S1}.

For $n = 0$, the $h$-cobordism theorem is trivially true: the interval is the only connected cobordism between connected $0$-manifolds.

For $n = 1$, the $h$-cobordism theorem is vacuously true, since there is no closed simply-connected $1$-dimensional manifolds.

For $n = 2$, the $h$-cobordism theorem is equivalent to the Poincar\'e conjecture stated 
by Poincar\'e in 1904 (one of the Millennium Problems) and was proved by Grigori Perelman in 
\cite{P1, P2, P3}.

For $n = 3$, the $h$-cobordism theorem for smooth manifolds has not been proved and, due to the $3$-dimensional 
Poincar\'e conjecture, is equivalent to the hard open question of whether the $4$-sphere has non-standard smooth structures.

For $n = 4$, the $h$-cobordism theorem was proved by Freedman and Quinn for topological manifolds \cite{fq}, 
but the analogous statement  in the $\mathrm{PL}$ and smooth categories was shown to be  false by  Donaldson \cite{D}.

The following is one of the most famous theorems in cobordism theory:

\begin{theorem}[$h$-cobordism]
Let $n$ be at least $5$ and let $W$ be a compact $(n + 1)$-dimensional $h$-cobordism between $M$ and $N$ in the category $\mathcal{C}=\mathrm{Diff}, \mathrm{PL}, \textup{or } \mathrm{Top}$ such that $W, M$ and $N$ are simply connected, then $W$ is $\mathcal{C}$-isomorphic to $M \times [0, 1]$. The isomorphism can be chosen to be the identity on $M \times \{0\}$.

This means that the homotopy equivalence between $M, W,$ and $N$ is homotopic to a $\mathcal{C}$-isomorphism.	
\end{theorem}

There ought to be many restrictions to families of deformations   in complex dimensions 3 and higher coming from this result, 
which apply to deformations of complex structures, and are certainly worth studying. In this initial study, we just list some very basic questions.

\begin{question}
Compare the notions of cobordism and deformations of compact manifolds. 
\end{question}

Cobordisms themselves may be considered in family and such families appear in the recent literature 
considered within  perspectives of various types of 
geometry  and topology  such as \cite[p.\thinspace 136]{IJP} and \cite[p.\thinspace 61]{ALP} 
as well as in current research in symplectic geometry, in particular in connection to Floer theory \cite{WW} 
and Fukaya categories of Lagrangian cobordisms \cite{BC, C}.
\begin{question}
Compare the notion of deformations of manifolds with boundaries with the notion of families of cobordisms.
\end{question}

\begin{question}
	Let $n\geq2$. If an $n$-dimensional complex manifold $M$ 
	with boundary has nontrivial  deformations, then  what type of cobordism exists between the components of $\partial M$?
\end{question}

If the assumption that $M$ and $N$ are simply connected is dropped, $h$-cobordisms need not be cylinders; the obstruction is exactly the Whitehead torsion $\tau(W, M)$ of the inclusion $M\hookrightarrow W$.

There exits a stronger version of the $h$-cobordism theorem, known as the $s$-cobordism theorem.
Namely, the $s$-cobordism theorem (the $s$ stands for simple-homotopy equivalence), proved independently by Barry Mazur, John Stallings, and Dennis Barden, states (with the same assumptions as above but where $M$ and $N$ need not be simply connected):
\begin{theorem}[$s$-cobordism]\label{s-cob}
An $h$-cobordism is a cylinder if and only if Whitehead torsion $\tau(W, M)$ vanishes.
\end{theorem}
The torsion vanishes if and only if the inclusion $M\hookrightarrow W$ is not just a homotopy equivalence, but a simple-homotopy equivalence.

Note that one need not assume that the other inclusion $N\hookrightarrow W$ is also a simple homotopy equivalence, 
it follows from the theorem.

Considering when $h$-cobordisms are trivial, we start our approach  towards the main question considered in Q. \ref{main-question}
concerning deformation theory in the non-compact case.
Meanwhile, 
returning to deformation theory, recall that the core results of Kodaira--Spencer theory state that for  a compact complex manifold $X$, the first cohomology group with coefficients in the tangent sheaf, that is $\mathrm{H}^1(X,TX)$, parametrizes infinitesimal deformations of $X$, with corresponding  obstructions given by elements of $\mathrm{H}^2(X,TX)$. In the case of cobordisms, the Whitehead torsion characterizes when an $h$-cobordism is a cylinder. Thus, it is natural to consider the following question.
	

\begin{question} Is there a numerical invariant that characterizes when a 
manifold with boundary is rigid, similarly to the way  Whitehead torsion characterizes 
when an $h$-cobordism is a cylinder?
\end{question}

Categorically, the collection of $h$-cobordisms forms a groupoid.
Then a finer statement of the $s$-cobordism theorem is that the isomorphism classes of this groupoid (up to $\mathcal{C}$-isomorphism of $h$-cobordisms) are torsors for the respective Whitehead groups $\mathrm{Wh}(\pi)$, where $ \pi \cong \pi _{1}(M)\cong \pi _{1}(W)\cong \pi _{1}(N)$.

\begin{question}
Express the definitions and properties of deformations of manifolds with boundaries in the language of 
groupoids and torsors. 
\end{question}

\section{Deformations of compactifiable manifolds}

 When we are only interested in an open subset $X$ of a compact complex manifold $W$  
 such that $D:= W\setminus X$ is a hypersurface, we may consider deformations of pairs $(W,D)$ such as in \cite{iac}.
  This may be too restrictive for the purpose of holomorphic deformations
   when we are only interested in $X$, even if  we know that $X$ 
  can be compactified; such a notion has been considered in the 
  literature \cite{Ka}, and here we give a slightly generalized version. 
   In our definition of  compactifiable deformation of the pair  $(X,W)$ 
  we do not require that the map be a submersion (nor that it be flat) at the points of $D$. 
  We inquire how restrictive are our assumptions and we pose a few questions related to the compactifiability of 
deformations of the threefolds $W_k$ studied in \cite{BGS,GKRS}.

\begin{definition}
	Take finitely many holomorphic functions $f_1,\ldots,f_k$ in a domain $D$ of $\mathbb{C}^n$ and form their ideal sheaf $$\mathcal{I}:=\mathcal{I}_D:=\mathcal{O}_Df_1+\cdots+\mathcal{O}_Df_k\subset\mathcal{O}_D.$$
	The quotient sheaf $\mathcal{O}_D/\mathcal{I}_D$ is a \textit{sheaf of rings} on $D$. Put
	$$X:=\mathrm{Supp}(\mathcal{O}_D/\mathcal{I}_D)=\{x\in D\colon\mathcal{I}_{D,x}\neq \mathcal{O}_{D,x}\},\quad \text{and} \quad \mathcal{O}_X:=(\mathcal{O}_D/\mathcal{I}_D)|_X.$$
	Clearly,  $X$ equals the set 
	$$N(\mathcal{I}):=N(f_1,\ldots,f_k)=\{x\in D\colon f_1(x)=\cdots=f_k(x)=0 \}$$
	 of common zeros of $f_1,\ldots,f_k$ in $D$ and, moreover, $X$ is \textit{closed in} $D$. Every stalk $O_{X,x}=(\mathcal{O}_{D,x}/\mathcal{I}_x)$ is an \textit{analytic algebra}, hence $\mathcal{O}_X$ is a sheaf of \textit{local $\mathbb{C}$-algebras} on $X$. The $\mathbb{C}$-ringed space $(X,\mathcal{O}_X)$ is called \textit{the (complex) model space, defined (on $D$) by $\mathcal{I}$}. We write $V(f_1,\ldots,f_k)$ or simply $V(\mathcal{I})$ for this space. Note that $|V(\mathcal{I})|=N(\mathcal{I})$.
\end{definition}
\begin{definition}
	A $\mathbb{C}$-ringed space $(X,\mathcal{O}_X)$ is called a \textit{complex space}, if $X$ is a Hausdorff \textit{space} and if every point of $X$ has an open neighborhood $U$ such that the \textit{open $\mathbb{C}$-ringed subspace $(U,\mathcal{O}_U)$} of $(X,\mathcal{O}_X)$ is isomorphic to a complex model space. Thus, locally, complex spaces are determined by \textit{finitely many} holomorphic functions in domains of number spaces.
\end{definition}
\textit{Complex model spaces} are complex spaces. In a complex space $(X,\mathcal{O}_X)$ every open set $U\subset X$ defines an \textit{open complex subspace $(U,\mathcal{O}_U)$}. Complex spaces form a (full) \textit{subcategory} of the category of $\mathbb{C}$-ringed spaces. Morphisms (isomorphisms) are called \textit{holomorphic} (\textit{biholomorphic}) maps; $\mathcal{O}_X$-modules on a complex space $(X,\mathcal{O}_X)$ are called \textit{analytic sheaves on $X$.}
\\
\textit{Algebraic} properties of the analytic algebras $\mathcal{O}_{X,x}$ are used to introduce \textit{geometrical} notions. For instance, we say that \textit{smooth} points are those points $x\in X$ for which $\mathcal{O}_{X,x}$ is regular, i.e. isomorphic to a $\mathbb{C}$-algebra $\mathbb{C}\{z_1,\ldots,z_n\}$. Complex spaces with smooth points only are called \textit{complex manifolds}; standard examples are domains in $\mathbb{C}^n$ and Riemann surfaces. Non smooth points are called \textit{singular} points;
simple examples of singular points are the origin of Neil's parabola $y=ax^{3/2}, a\neq 0$ constant, and the origin of the cone $w^2-z_1z_2=0$ in $\mathbb{C}^3$.

\begin{definition}\label{e4}
Let $X$ be a complex space. We say that $X$ is {\it compactifiable} if there are a compact complex space $X'$, a closed analytic
subset $D\subset X'$ and an open embedding $j\colon X\to X'$ such that $j(X)=X'\setminus D$. If in addition we can take $X'$ 
to be a projective variety with $X$ embedded analytically on $X'$
we call it a  {\it projectivizable deformation} of $X$.
\end{definition}

Note that in definition \ref{e4} we do not fix the pair $(X',D)$. Note also that  requiring an algebraic embedding instead 
of an analytic one would result in a different definition.  By Hironaka desingularization 
 if $X$ is a smooth compactifiable complex space, then we may find a compactification
$(X',D,j)$ with $X'$ smooth and $D$ with normal crossings. Note also that by definition 
quasi-projective varieties are compactifiable, but not conversely.

\begin{definition}\label{e5}
Let $X$ be compactifiable complex space. A \emph{compactifiable deformation of $X$} is a sextuple $(\mathcal{X},D,Y,f,o,j)$, where:
\begin{itemize}
\item $\mathcal X$ and $Y$ are complex spaces,  
\item $D$ is a closed analytic subset of $\mathcal X$ and $o\in Y$, 
\item $f\colon  \mathcal X\to Y$ is a proper holomorphic map such that $f_{|\mathcal X\setminus D}$ is
flat,
\item  
$j\colon  X\to f^{-1}(o)$ is an open embedding with
$$j(X) =f^{-1}(o)\setminus D\cap f^{-1}(o).$$ 
\end{itemize}
We say $(\mathcal{X},D,Y,f,o,j)$ is \emph{topologically trivial along $X$}
if $$f_{\mathcal X\setminus D}\colon  \mathcal X\setminus D\to Y$$ is topologically locally trivial over $Y$, 
i.e. for each $u\in Y$ there is an
open neighborhood $A$ of $u$ such that $f_{(\mathcal X\setminus D)\cap f^{-1}(A)}\colon   (\mathcal X\setminus D)\cap f^{-1}(A)\to A$
and the projection $(\mathcal X\setminus D)\cap f^{-1}(A)\times A\to A$ are isomorphic as maps of topological spaces over $A$. 

When
$X$ is smooth, we say that
$(\mathcal X,D,Y,f,o,j)$ is a
\emph{smooth compactifiable deformation} if $f_{|\mathcal X\setminus D}$ is a smooth morphism of complex spaces and each fibre of
$f$ is a complex manifold. If $\mathcal X$ can be chosen to be a  projective variety we say that the deformation is \emph{projectivizable}.
\end{definition}

For a smooth compactifiable deformation we do not assume that $Y$ is a complex manifold. This is the main 
reason to use
topological local triviality in Definition \ref{e5} instead of differential triviality. The interested reader may do the
definition for differential local triviality when $Y$ is smooth.

\begin{example}
Let $X$ be a connected compact manifold of dimension $n>0$. Let $S\subset X$ be a nonempty  finite set and $U=X\setminus S$.
Let  $(\mathcal X,D,Y,f,o,j)$ be a smooth compactifiable deformation of $X$. Since $f$ is proper and the deformation of $X$ is
smooth, all fibres of $f$ are compact manifolds with deformations of $U$ as Zariski open subsets. Restricting $Y$ 
 to a smaller neighborhood of $o$ if needed,  we may assume that all fibres of
$f$ are connected $n$-dimensional complex manifolds and $f_{|D}\colon  D\to Y$ is a finite holomorphic map. Thus all fibres of
$f_{|D}$ are finite sets. It is easy to check that $(\mathcal X,D,Y,f,o,j)$ is topologically trivial along $X$ if and only if all
fibres of $f_{|D}$ have the same number of points.
\end{example}

\begin{question}
Let $X'$ be a compactification of the complex space $X$. 
When are deformations of $X$  obtained from deformations of $X'$ and deformations of the analytic set $D\ce X'\setminus X$?
\end{question}

We  now discuss a few basic examples.

\begin{example}
Let $C$ be a smooth and connected complex curve of genus $g\ge 0$. Take as $D$ a finite set,  $X:= C\setminus D$.
 Set $m:= \sharp (D)$. In such a case deformations of $X$ topologically trivial over the base always come from deformations of $C$. 
 If either $g\ge 2$, or $g=1$ and $m>0$, or else $g=0$ and $m\ge 3$, then we are just speaking about the deformation theory of pairs 
 $(C,D)$, which is the (well-known) deformation theory of elements of $\Mm _{g,m}$. 
\end{example}

\begin{example}\label{Zkexample}
	Let us a consider the following noncompact surfaces
	$$Z_k=\Tot(\mathcal{O}_{\mathbb{P}^1}(-k)).$$
	That is, total spaces of line bundles on $\mathbb{P}^1$. In \cite[Thm.\thinspace5.4 and Thm.\thinspace6.18]{BG}, the authors prove that these surfaces admit smooth compactifiable deformations. In fact, \cite[Lem.\thinspace5.6]{GKRS} show that these deformations can be obtained from deformations of the well known Hirzebruch surfaces $$\mathbb{F}_k=\Proj(\mathcal{O}_{\mathbb{P}^1}(-k)\oplus\mathcal{O}_{\mathbb{P}^1}(0)),$$ 
	which are compact.
\end{example}

\begin{example}\label{Wkexample}
	The noncompact threefolds 
	$$W_k=\Tot(\mathcal{O}_{\mathbb{P}^1}(-k)\oplus\mathcal{O}_{\mathbb{P}^1}(k-2))$$
	admit deformations which are not smooth compactifiable. In \cite[Thm.\thinspace6.3]{GKRS} and \cite[Cor.\thinspace1.29 ]{BG} there are concrete examples of infinite dimensional families of deformations of $W_k$ for $k\geq2$ which are not smooth compactifiable. Further, in \cite[Cor.\thinspace1.25]{BGS} the authors show that deformations of $W_k$ are not obtained from their compactifications. This situation is quite different from the one of the surfaces $Z_k$ considered in the previous example.
\end{example}

\begin{question}
Determine which deformations of $W_k$ are compactifiable. 
\end{question}

\begin{question}\label{qq2}
Describe all projectivizable deformations of  $W_k$. 
\end{question}

\begin{question}
Decide whether there exist nontrivial deformations of $W_1$. Note that 
\cite{GKRS} proved that $W_1$ is formally rigid. 
\end{question}

\begin{question}\label{qq3}
Consider a compactifiable (or projectivizable) deformation $(\mathcal X,D,Y,f,o,j)$   of 
 $X=W_k$. For which $t\in Y$ is there a surjection
${f}^{-1}(t)\to \PP^1$ extending the structure of $\Xx \cap {f}^{-1}(t)$ as 
an affine bundle over $\PP^1$? For which $t\in Y$ is  ${f}^{-1}(t)$ at least covered by 
codimension $1$ projective spaces?
\end{question}

\begin{question}
Does there exist a compactifiable space $X$ such that no nontrivial  deformation of $X$ is compactifiable? 
\end{question}

\begin{question}
Let $X$ be a toric Calabi--Yau threefold. Determine what conditions on $X$ imply that 
all of its deformations are compactifiable (if any).  
\end{question}


\subsection{Nontrivial examples in  dimension one}
To give some very nontrivial examples, we recall some definitions and results from \cite{e}.

\begin{definition} Let $f\colon X\to Y$ be a holomorphic submersion between complex manifolds. $f$ is said to be 
\emph{differentially  locally trivial over $Y$} (resp. \emph{locally trivial over $Y$}) if for each $y\in Y$ there is an open neighborhood $U$
of $y$ such that the map $f_{|f^{-1}(U)}\colon f^{-1}(U)\to U$ 
is differentially (resp. holomorphically) isomorphic to the projection $f^{-1}(y)\times U\to U$.
\end{definition}

If the submersion  $f$ is proper, then it  is differentially  locally trivial over $Y$ by the classical theorem of Ehresmann. 
As mentioned in the introduction, the  complex analogue is false.
Furthermore, the theorem of Fischer and Grauert   \cite{fg} says that if the fibers are compact connected complex manifolds,
then $f$ is holomorphically locally trivial if and only if all the fibers $f^{-1}(y)$ are analytically isomorphic for all $y\in Y$.
The hypothesis of compact fibers is essential here, as  \cite[Ex.\thinspace7]{GR} demonstrates.

\begin{theorem}
\cite[Thm.\thinspace III]{e} \label{e2} If $X$ is a smooth surface, $Y$ a smooth curve, all fibres of $f$ are biholomorphic to an algebraic curve $F$ and $f$ is compactifiable, then $f$ is locally trivial over $Y$.
\end{theorem}


We recall that a connected one-dimensional complex manifold
$F$ is of finite type if and only if there is a compact complex one-dimensional manifold $T$ and a holomorphic embedding $j\colon F\to T$ such that $\overline{j(F)}\setminus F$
is topologically the union of finitely many disjoint circles and points. The following very strong result is due to K. Yamaguchi: 

\begin{theorem}\label{y1}
\cite[Thm.\thinspace3]{y} Let $f\colon W\to Y$ be a holomorphic submersion with $Y$ a one-dimensional complex manifold, $W$ a Stein surface and all fibres of $f$ biholomorphic to the same connected one-dimensional complex manifold with finite type $F$. Assume that $F$ is neither biholomorphic to the open disc nor to the open disc minus its center.
Then $f$ is holomorphically a locally trivial fibration over $Y$.
\end{theorem}

Thus here one does not need to assume that the submersion is compactifiable, but \cite{y} requires 
a very strong assumption (Steinness) on the total space of the holomorphic submersion $f$. 
We just recall the following example from \cite{e}. Other examples of nontrivial holomorphic fibrations 
 are contained in \cite{y} and the references quoted therein.

\begin{example}\label{e3}
\cite[Ex.\thinspace1.2.1]{e} Set $T := \{(0,0)\}\cup \{(z,z^{-1})\}_{z\in \CC^\ast} \subset \CC^2$. Set $X:= \CC^2\setminus T$. Let $f: X\to \CC$ denote the restriction to $X$ of the projection of $\CC^2$ onto its factors. All fibres of the submersion $f$ are biholomorphic to $\CC^\ast$, but $f$ is neither differentiably nor topologically locally trivial over $\CC$.
Note that $X$ is not Stein.
\end{example}

\section{Deformations of $q$-concave spaces}
We use the notation from \cite{hl}. For  extensions of these definitions to complex spaces, see for instance \cite[Ch.\thinspace5]{and2}. 
Let $X$ be a connected complex manifold of dimension $n$ and   $x\in X$. 
Then $X$ is a $2n$-dimensional differentiable manifold and we call $T_{\RR,x}X$ its real tangent space, 
which is a $2n$-dimensional real vector space. Set 
$$T_{\CC,x}X:= T_{\RR,x}X\otimes \CC.$$ 
Thus $T_{\CC,x}X$ is a $2n$-dimensional complex vector space. If $z_1,\dots ,z_n$ are local coordinates of $X$ around $x$, 
 we may take $\frac{\partial}{\partial z_1},\dots ,\frac{\partial}{\partial z_n},\frac{\partial}{\partial \bar{z}_1}, \dots, \frac{\partial}{\partial \bar{z}_n}$
as a basis of $T_{\CC,x}X$. Let $T'_xX$ (resp. $T''_xX$) 
be  the subspace of $T_{\CC,x}X$ spanned by $\frac{\partial}{\partial z_1},\dots ,\frac{\partial}{\partial z_n}$ (resp. $\frac{\partial}{\partial \bar{z}_1}, \dots, \frac{\partial}{\partial \bar{z}_n}$) \cite[p.\thinspace59]{hl}.  If $Y\subset X$ is a real $C^1$ submanifold of $X$ 
 we set
$T'_xY:= T'_xX\cap T_{\CC,x}Y$.

\begin{definition}
Let $\rho\colon X\to \RR$ be a $C^2$ function. For any $x\in X$ we define a
 Hermitian form $L_\rho(x)$ on $T'_xX$, called the \emph{Levi form} of $\rho$,
in the following way. If $(t_1,\dots ,t_n)\in \CC^n$ and $t = \sum _{j=1}^{n} t_j\frac{\partial}{\partial z_j}$ set
$$ L_\rho (x)(t) := \sum _{j,k=1}^{n} \frac{\partial ^2\rho (x)}{\partial \bar{z}_j\partial \bar{z}_k}\bar{t}_jt_k.$$
\end{definition}

The signature (also called the inertia) of the Levi form at $x$ does not depend on the choice of the holomorphic local coordinates 
$z_1,\dots ,z_n$,
see \cite[pp.\thinspace59-60]{hl}, \cite[p.\thinspace61]{and2}.

\begin{definition}
Fix an integer $q$ such that $1\le q\le n$. We say that a $C^2$ function $\rho$ is \emph{$q$-convex} if at all $x\in X$ 
the Levi form of $\rho$ has at least $q$ positive eigenvalues \cite[p.\thinspace60]{hl}
and that it is \emph{$q$-concave} if $-\rho$ is $q$-convex, i.e. if at all points of $X$ 
the Levi form of $\rho$ has at least $q$ negative eigenvalues \cite[p.\thinspace73]{hl}. 

A $C^2$-function $\rho$ is called \emph{pseudoconvex} (resp \emph{pseudoconcave}) if it is $(n-1)$-convex (resp. $(n-1)$-concave).
The  $n$-convex functions are those usually called strictly plurisubharmonic.
\end{definition}

\begin{definition}
Let $\rho\colon  X\to \RR$ be a $C^2$ function. Set $\beta := \sup _{z\in X} \rho(z)  \in \RR\cup \{+\infty\}$. The function $\rho$ is called an \emph{exhausting function} if there is a compact subset $K\subseteq X$ such that
$\rho (x) < \beta $ for all  $x\in X\setminus K$ and for all $c<\beta$ the set $\{z\in X\mid \rho (z)<c\}$ is relatively compact in $X$
\cite[I.5.1]{hl}. 

The exhausting function $\rho$ is called \emph{$q$-concave at infinity} if there is a compact $K\subset X$ 
with $X \setminus K \neq \emptyset$ and such that $\rho _{|X\setminus K}$ is $q$-concave.
\end{definition}

\begin{definition}
 We say that an $n$-dimensional  complex manifold
$X$ is \emph{$q$-concave} (resp. $q$-convex) for  $0\le q\le n-1$, if there exists an exhausting function of $X$, which is $(q+1)$-concave 
(resp. $q$-convex) at infinity \cite[p.\thinspace73, p.\thinspace 65]{hl}.
\end{definition}

Now assume that $X$ is a connected open subset of the connected complex manifold $W$. Fix $q\in \{1,\dots ,n-1\}$. In
\cite[pp.\thinspace138--139]{hl} it is defined when $W$ is a $q$-concave  extension of $X$.  We recall the basic idea of the definition, 
referring the reader to \cite{hl} for technical details. A $q$-concave extension in $X$ is defined locally in domains, in the same spirit 
as the manifold is obtained by local charts. The choices of domains $V \subset X$ are typically quite delicate and require satisfying a series 
of rather technical conditions.  

\begin{definition}
Assuming  a domain $V\subset X$ has been appropriately chosen,  consider 
$A_1 \subset A_2$ such that  $A_2\setminus A_1 \subset V$. We say that $A_2 $ is  a \emph{$q$-concave extension} of $A_1$ in $X$
provided there exists a local chart $\psi\colon V \rightarrow  U \subset \mathbb C^n$ such that  $U$ is a $q$-concave domain in $\mathbb C^n$
and there exists a $q+1$-concave function $\rho\colon U \rightarrow \mathbb R$ of class $C^2$
such that $\psi(A_2\setminus A_1)= \{z \in U: \rho(z) <0\}$ 
is relatively compact in $U$ and forms a $q$-concave configuration in $\mathbb C^n$ (see \cite[Def.\thinspace 13.1]{hl}).
\end{definition}

Intuitively, we may think of a $q$-concave manifold as being concave in $q$ dimensions, and similar for $q$-convex. 
Concave and convex extensions can be thought of as being represented by the following pictorial situation. 
Suppose a physical space $X$ is an extension of a proper subspace $D \subset X$ and an observer is positioned in 
$X\setminus D$, then: 
\begin{enumerate}
\item[(Fig.\ref{fig:q-concave-extension})] 
If $X$ is a concave extension of $D$, then the observer sees $D$  as concave. 
\item[(Fig.\ref{fig:q-convex-extension})] 
If $X$ is a convex extension of $D$, then the observer sees  $D$  as convex. 
\end{enumerate}
Note however that the notions of $q$-concave and $q$-convex extensions of $D$ are concepts 
referring to a ``small" neighborhood of $D$ and do not imply that $X$ is either concave or convex itself,
as the pictures show.

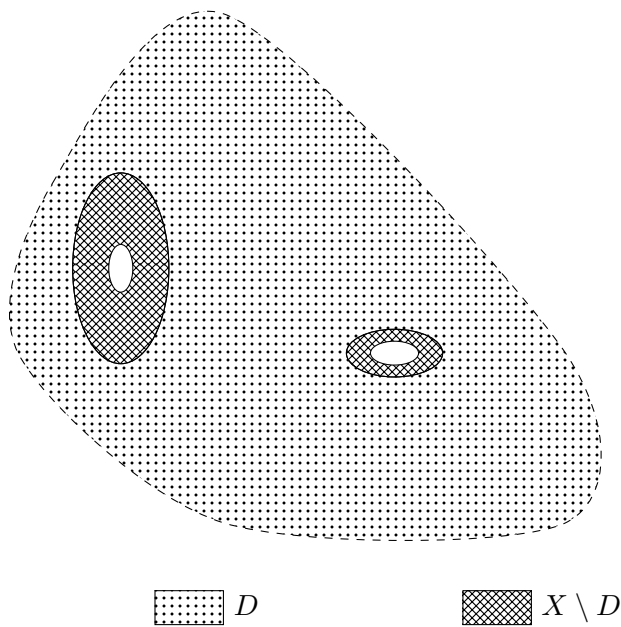
\begin{figure}
\begin{tikzpicture}[scale=4.5]
\filldraw[dashed,draw=black,pattern=dots]
plot[smooth cycle] coordinates{
	(0.4,1)(0.5,1.5)(1,2)(2,1)(2,0.5)(1,0.5)};
\draw[fill=white] (0.7,1.25) ellipse (4pt and 8pt);
\draw[fill=white] (1.5,1) ellipse (4pt and 2pt);
\draw[pattern=crosshatch] (1.5,1) ellipse (4pt and 2pt);
\draw[fill=white] (1.5,1) ellipse (2pt and 1pt);
\draw[pattern=crosshatch] (0.7,1.25) ellipse (4pt and 8pt);
\draw[fill=white] (0.7,1.25) ellipse (1pt and 2pt);
\filldraw[pattern=dots] (0.8,0.2) rectangle (1,0.3);
\filldraw[pattern=crosshatch] (1.7,0.2) rectangle (1.9,0.3);
\draw (1,0.25) node[right] {$D$};
\draw (1.9,0.25) node[right] {$X\setminus D$};
\end{tikzpicture}
\caption{$X$ is a $q$-concave extension of $D$.}
\label{fig:q-concave-extension}
\end{figure}
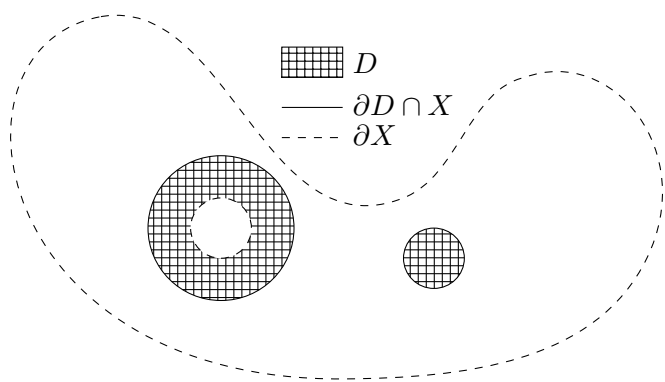
\begin{figure}
\begin{tikzpicture}[rotate=270,scale=0.08]\draw[dashed] (0, 0) .. controls (5.18756, -26.8353) and (60.36073, -18.40036)
.. (60, 40) .. controls (59.87714, 59.889) and (57.33896, 81.64203)
.. (40, 90) .. controls (22.39987, 98.48387) and (4.72404, 84.46368)
.. (10, 70) .. controls (12, 60.7165) and (26.35591, 59.1351)
.. (30, 50) .. controls (39.19409, 26.95198) and (-4.10555, 21.23804)
.. (0, 0);
\draw (35,20) circle (12);
\fill[pattern=horizontal lines] (35,20) circle (12);
\fill[pattern=vertical lines] (35,20) circle (12);
\fill[fill=white] (35,20) circle (5);
\draw[dashed] (35,20) circle (5);
\draw (40,55) circle (5);
\fill[pattern=horizontal lines] (40,55) circle (5);
\fill[pattern=vertical lines] (40,55) circle (5);
\draw [pattern=grid] (5,40) rectangle (10,30);
\draw (8,40) node[right] {$D$};
\draw (15,30) -- (15,40); 
\draw (15,40) node[right] {$\partial D\cap X$};
\draw [dashed](20,30) -- (20,40); 
\draw (20,40) node[right] {$\partial X$};
\end{tikzpicture}
\caption{$X$ is a $q$-convex extension of $D$.}
\label{fig:q-convex-extension}
\end{figure}

\begin{remark}\label{and1}
Let $W$ be a connected complex manifold that is a $q$-concave extension of its connected open subset $X\subset W$.
 Assume that $X$ is $q$-concave. Then the
restriction map $H^i(W,TW) \to H^i(X,TX)$ is an isomorphism for $i\le q-1$  by \cite[IV.15.11]{hl} and  is
injective if $i= q$ by \cite[IV.16.1]{hl}, since the
Dolbeault and  the usual cohomology groups coincide by
\cite[Thm.\thinspace I.2.14]{hl}.
\end{remark}

\begin{definition}
An $n$-dimensional $(n-1)$-concave complex manifold is  called a \emph{pseudoconcave} complex manifold.
\end{definition}

\begin{example}
Every compact complex space is trivially pseudoconcave. 
 More examples are easily obtained by  taking ``convex holes'' in a compact space.
The compactification problem (solved by H. Rossi) says that if a space $X$ is pseudoconcave and $\dim X\ge 3$, 
then $X$ is an open subset of a compact space (it may be singular) \cite[\S\thinspace23]{hl} which is an  
$(n-1)$-concave extension of $X$, hence one can ``fill in" convex holes in  $\dim \ge 3$.
\end{example}

\begin{example}
Not every $2$-dimensional pseudoconcave space is compactifiable. Indeed,
\cite[Ex.\thinspace3]{and2}
shows how to construct a pseudoconcave structure on $\mathbb P^2\setminus \{p\}$ which is not compactifiable. 
 \end{example}

Thus, it is not very restrictive (but restrictive) to only look at $q$-concave complex manifolds with a compact $q$-concave extension.
Here we take a pair $(W,X)$ with $X$ $q$-concave, $W$ a compact complex manifold and we discuss 
smooth $q$-concave deformations. For simplicity, in what follows we assume that the parameter space $Y$ is smooth 
and also that the total space of the deformation $\Xx$ is smooth.

\begin{definition}
A  \emph{ $q$-concave deformation} of the pair $(W,X)$ is a quintuple $(\Xx,\Ww,Y,o,f)$,  satisfying:
\begin{itemize}
\item $Y$ is a complex space, and $o\in Y$, 
\item $f\colon \Ww \to Y$ is a proper submersion, 
\item $\Xx$ is an open subset of $\Ww$, 
\item $(f^{-1}(o), f^{-1}(o)\cap \Xx)$ is biholomorphic to the pair $(W,X)$, and 
\item $f^{-1}(y)$ is a $q$-concave extension of $f^{-1}(y)\cap \Xx$  for each $y\in Y$.
\end{itemize}
\end{definition}

The last requirement means that 
$f_{|\Xx}$ is transversally $q$-concave, that is, 
 $\Ww$ is a $q$-concave extension of $\Xx$ in which the data of the $q$-extension is transversal to $f$.

When $f_{|\Xx}$ is topologically locally trivial, we say that the deformation is \emph{topologically  locally 
trivial} over $Y$. When $Y$ is a smooth manifold and   $f_{|\Ww}$ is differentially locally trivial over $Y$
we say that the deformation is \emph{differentially locally trivial} over $Y$.

Now assume $q\ge 3$. By Remark \ref{and1} the restriction map 
$$H^i(W,TW) \to H^i(X,TX)$$ is an isomorphism for each $i\leq 2$ and is an isomorphism in degree $3$ if $q>3$.
 Since these are the cohomologies used in  Kodaira--Spencer theory, it then follows that  of $q \geq 3$ 
 the deformation theory of $W$ is governed by data on $X$. 
 For instance, if $H^1(X,TX)=0$, then in any $q$-concave deformation $(\Xx,\Ww,Y,o,f)$, $\Ww$ is locally trivial over $Y$
and hence we just ask how to deform $X$ inside $W$ (if we only want $q$-concave deformations of $(W,X)$).

Take a $q$-concave deformation $(\Xx,\Ww,Y,o,f)$ of the pair $(W,X)$ with $Y$ smooth and let $E$ be a holomorphic vector bundle on $X$. 
By \cite[Thm.\thinspace3 \S\thinspace8  p.\thinspace123]{RaRu}
there is a range of integer $i\in \NN$,  $q$  and  $\dim Y$ for which the direct images $R^if_\ast (E)$ are coherent. 
However, to study deformations, we only need the case $E =T\Xx$
and in this case $ E$ is the restriction to $\Xx$ of the vector bundle $T\Ww$. Since $f\colon \Ww \to Y$ is proper,
 Grauert theorem of coherence for proper maps says that for all $i\in \NN$ the $\Oo_Y$-sheaf, $R^if_\ast (T\Ww)$ is coherent, each cohomology group $H^i(f^{-1}(y),Tf^{-1}(y))$ is finite-dimensional and the function $y\mapsto \dim H^i(f^{-1}(y),Tf^{-1}(y))$
is an upper semicontinuous function $Y\to \NN$.  We then obtain the following result:

\begin{proposition} Let $(W,X)$ be a $q$-concave pair, $q\ge 3$. Then all $q$-concave deformations $(\Xx,\Ww,Y,o,f)$ of  $(W,X)$ with $Y$ smooth are obtained from deformations of $X$. 
\end{proposition}

\begin{proof} Using Remark \ref{and1}
we have 
$$ \dim H^i(f^{-1}(y),Tf^{-1}(y)) =  \dim H^i(f^{-1}(y)\cap \Xx,Tf^{-1}(y)\cap \Xx)$$ for all $i\le q-1$.
\end{proof}

Let $X$ be an $n$-dimensional connected complex space, $n\ge 3$. For simplicity we assume here that $X$ is smooth and relegate
to Remark \ref{sing1} the case of singular complex spaces. 
Assume that
$X$ is pseudoconcave, i.e. that it is $(n-1)$-concave \cite{and1,and2,hl}. 

\begin{remark}
Even for projective manifolds $X$ it is 
classically known (e.g. K3 surfaces, i.e. a compact $2$-dimensional Calabi-Yau manifold)  that small deformations of
$X$ may not be projective.
\end{remark}

 The usual way to remain in the category of projective manifolds is  not just to consider deformations of
$X$, but instead to consider deformations of a pair $(X,L)$, where $L$ is an ample line bundle on $X$. In this section we will do the
same for pseudoconcave complex manifolds. We follow \cite{and2,as,at}. 

Assume the existence of  a line bundle $L$ on $X$
such that for each $x\in X$, there is a positive integer $m$ and $f_1,\dots ,f_n\in H^0(X,L^{\otimes m})$ such that $f_i(x)=0$
for all $i$ and
(after taking a local trivialization of $L$ around $x$), $f_1,\dots ,f_m$ induce a local biholomorphism of $X$ near $x$
with $\CC ^n$ near $0$.  When $n=2$ one needs also to  assume that for all $x, y\in X$ such that $x\ne y$ there is a positive
integer $m$ and $f\in H^0(X,L^{\otimes m})$ such that $f(x)=0$ and $f(y) \ne 0$. Under these  assumptions
the $\CC$-algebra 
$$A(X,L) := \oplus _{m\ge 0} H^0(X,L^{\otimes m})$$
 is finitely generated and hence the scheme
$$\Gamma:= \mathrm{Proj}(A(X,L)$$
 is a well-defined complex projective scheme \cite{hart}. Since $L$ is a line bundle on $X$
and it is known that for each $x\in X$ there is a positive integer $m$ and $f\in H^0(X,^{\otimes m})$ with $f(x)\ne 0$, the
pair $(X,L)$ induces a holomorphic map 
$$u\colon X\to \Gamma.$$ 
Then $X$ is an integral complex space (for manifolds it just means
that $X$ is connected). Since $X$ is smooth, it is a normal complex space \cite[Thm.\thinspace p.\thinspace120]{fischer}. 
Let $v\colon W\to
\Gamma$ be the normalization of
$\Gamma$ as a complex space \cite[\S\thinspace2.26]{fischer}, which is just the complex space associated to the normalization of
$\Gamma$ as an integral algebraic variety. We get that
$u$ factors as
$$u = v\circ j  \quad \text{with}  \quad j\colon X\to W$$
 \cite[pp.\thinspace121-122]{fischer}. It is proved in \cite{and2,as,at} that $j$ is
an open embedding and in the following we will identify $j$ with the identity map, i.e. we will see $X$ as an open subset of
the normal projective complex space $W$. Examples contained in the quoted paper show that for $X$ smooth $W$ need not be
smooth, but it has at most finitely many singular points. 

By the construction of
$\Gamma$ as a $\mathrm{Proj}$ there exists a natural line bundle on it and its pull-back $R$ to $X$ has the property that it is
ample and $R_{|X} =L$. Taking instead of $L$ a power $L^{\otimes m}$, $m\gg 0$, we may assume that $R$ induces an embedding of
$W$ into a large projective manifold. Using the definition of projective morphisms between two complex spaces
\cite[Ch.\thinspace4]{bs2} we may define projective deformations of the pair $(X,R)$ in the following way.
Take $W$ and $R$ as above. 

\begin{definition}
A {\it projective deformation} of $(X,L,W,R)$ is given by $(\Ww,Y,f,o,\Xx, \Rr)$ such that:
 \begin{itemize}
\item  $\Ww$ and $Y$ are complex spaces,
\item  $o\in Y$ and $f^{-1}(o) =W$,
\item $f\colon 
\Ww
\to Y$
is a holomorphic submersion which is projective in the sense of 
\cite[Ch.\thinspace4]{bs2},
\item  $\Xx$ is an open subset of $\Ww$,
\item  $f_{|\Xx}\colon \Xx \to Y$ is pseudoconcave,
\item $f^{-1}(o)\cap \Xx =X$, 
\item $\Rr$ is a line bundle on $\Ww$ which is ample with respect to $f$ and 
\item $\Rr _{|f^{-1}(o)}
\cong R$.
\end{itemize}
\end{definition}

Note that by definition $f$ is projective if
and only if there is a line bundle on $\Ww$ which is ample with respect to $f$  \cite[p.\thinspace 159]{bs2}.

\begin{remark}
Assume $n\ge 3$ and hence that $W$ is a $2$-concave extension of $X$. Assume that $W$ is smooth and that $H^1(TX)=0$. Since
$H^1(W,TW) =H^1(X,TX)$, $W$ is rigid. We get a kind of rigidity for pseudoconcave deformations of the pair $(X,L)$: the only
nearby pseudoconcave deformations are as open subsets of $W$ for the following reason. 
Since $W$ is rigid, any small deformation of $W$ is biholomorphic to the projection 
$\pi _2\colon W\times \Delta \to \Delta$ with $\Delta$ an open neighborhood of $0\in \CC^N$, 
some $N\ge 0$. By the definition of pseudoconcave deformation such a small deformation is obtained taking an open subset $\Xx$
of $W\times \Delta$ such that $\Xx \cap \pi _2^{0} =W$ and $\pi _{2|\Xx} \Xx \to \Delta$ is a submersion.
 Thus, this submersion is equivalent to taking a family $\{\pi _2^{-1}(z)\cap \Xx\}_{z\in \Delta}$ of open subsets of $W$ with $\pi _2^{-1}(0) =X$. \end{remark}

Now we may take as a pseudoconcave deformation of $(X,L)$ any family  with $\Xx$ an open subset of $\Ww$ such that $\Xx \cap
f^{-1}(o) =X$ and $f_{|\Xx}$ is pseudoconcave.

\begin{remark}\label{sing1}
For general background on singular complex spaces, see \cite{fischer}. See \cite[\S\thinspace2.6, 2.7]{fischer} for the differential of a
holomorphic map between arbitrary complex spaces and in particular for the notion of submersion between complex spaces
\cite[\S\thinspace2.18]{fischer}. For the notions of $q$-convexity and $q$-concavity for complex spaces, see for instance \cite[\S\thinspace
5.2]{and2}
\end{remark}

\begin{question}\label{qq1}
Let $X$ be a pseudoconcave complex space. When is there a compact complex space $W$ and an open embedding 
$j\colon X\to W$ such
that $W\setminus j(X)$ is ``very small'', for instance a finite set?
\end{question}

One could think of  Question \ref{qq1} as a particular case of the following vague question.

\begin{question}\label{vague}
Let $X$ be a complex space. Find an object, say a compact topological space, 
$\overline{X}$ containing $X$ with $X$ open in $\overline{X}$ and for which certain 
sequences (or all sequences) of points of $X$ have a subsequence converging to a point of 
$\overline{X}$, then give a structure of complex space to $\overline{X}$. 
Classify those complex spaces $X$ such that $\overline{X}\setminus X$ is small, e.g. it is a finite set. 
\end{question}

Such a  program was carried out  for  noncompact one-dimensional complex manifolds in several ways,
 using non-holomorphic data, like the existence or non-existence of certain potential \cite[pp.\thinspace20--21]{rs}, 
\cite{ahs}, at least if $\pi _1(X)$ is finitely generated 
 then giving several different criteria which determine when $\overline{X}\setminus X$ is small.

 When $n:= \dim X\ge 2$ and $X$ is smooth, a related notion is the one of envelope of holomorphy, 
 in which we assume that $X$ has many nonconstant holomorphic functions and
  the aim is to find $\overline{X}$ which is Stein. This very old problem (see \cite{jp} for the case of Riemann domains) is discussed
in details in \cite[\S\thinspace11.7.3, \S\thinspace13.7]{dsst}.

We consider here the case of  a compact projective $W$ which is 
normal with only finitely many singular points  (this is similar to \cite[case (b) of Proposition 6]{at}). 
Then we expect that to verify when $W\setminus j(X)$ is finite should not be very difficult, 
although we do not have an obvious criterion to see it. We have:

\begin{proposition}\label{0conv}
Let $X$ be an $n$-dimensional pseudoconcave complex manifold. Assume that one of the following conditions
is satisfied:
\begin{enumerate}
\item either $n\ge 3$ and for each $x\in X$ there exists an integer $t>0$ and $f_1,\dots ,f_t\in \omega _X^{\otimes t}$ 
giving local coordinates at $x$; or
\item for each $x\in X$ there exists an integer $t>0$ and $f_1,\dots ,f_n\in \omega _X^{\otimes t}$ giving local coordinates of $X$ at 
$x$ and for all $x, y\in X$ such that $x\ne y$ there exists
$m>0$ and $f\in H^0(\omega _X^{\otimes m})$ such that $f(x)=0$ and $f(y)\ne 0$.
\end{enumerate}
Then there exists a normal projective variety $W$ and an open embedding $j\colon X\to W$.
\end{proposition}
\begin{proof}
Part (2) is \cite[Thm.\thinspace2]{at} or \cite[Thm.\thinspace4.3.1]{and2}. Part (1) follows from part (2) and \cite[Thm.\thinspace4.4.1]{and2}.
\end{proof}

We then expect that the following instance of the vague question phrased in \ref{vague} should be solvable in reasonably short time. 
\begin{question}
Solve question \ref{vague} for spaces $X$ satisfying the properties in Proposition \ref{0conv}.
\end{question}

\section{Final remarks}

We have stated a few concrete open questions on deformation theory of noncompact manifolds, of compactifiable manifolds,
manifolds with boundary, and $q$-concave spaces. Each of the sections $3$, $4$, and $5$ may give rise to 
an entire graduate thesis, and certainly to new publications. Also, each question about deformations has a corresponding 
question about moduli which should also be addressed. 
Moreover, there are many large fundamental questions about deformations that remain open, we finish by stating  a few of 
them.

\begin{question}
Under what conditions do all deformations of the total space of a vector bundle have the structure of affine bundles?
\end{question}

\begin{question}
Let $X$ be a noncompact Calabi--Yau threefold.  What geometric 
conditions on 
 $X$ imply that  it has infinitely  many nonisomorphic deformations?
(Compare with \cite{BGS} and \cite{GKRS}).
\end{question}

\begin{question}
Let $X$ be obtained from $\tilde X$ by removing finitely many   points. 
Compare  deformations of $X$ to those of $\tilde X$.
\end{question}

Finally, perhaps the question that must be answered most urgently:

\begin{question}
What is the correct cohomology theory that parametrizes deformations of noncompact manifolds?
(Vanishing of such a cohomology should imply  existence of   only trivial deformations.)
\label{main-question}
\end{question}

\section{Acknowledgements} 
This worked started during a meeting of Research in Pairs supported by the
Centro Internazionale per la Ricerca Matematica, of
Fondazione Bruno Kessler,  Trento (Italy).
F. R. was supported by 
Beca Doctorado Nacional Conicyt Folio 21170589.
E. G. and F. R. thank the Office of External Activities 
of ICTP for the support under
Network grant NT8.
E.B. was partially supported by MIUR and GNSAGA of INdAM (Italy). The authors thank Severin Barmeier for suggesting several improvements to the original text. Finally, we gladly acknowledge the help of the referee, who made us not only improve the paper, but  also learn more about the subject.

\end{document}